\documentclass[11pt]{amsart}
\usepackage[shortalphabetic]{amsrefs}
\usepackage[utf8]{inputenc}
\usepackage{amsthm}
\usepackage[utf8]{inputenc}
\usepackage{appendix}
\usepackage{cite}
\usepackage{amssymb,amsmath,amsthm}
\usepackage{hyperref}

\newcommand{\amsprimary}[1]{{\footnotesize\noindent AMS 2020 \textit{Mathematics subject
classification:} Primary #1\vspace{1pc}}}
\newcommand{\keywordsnames}[1]{{\footnotesize\noindent\textit{Key words:} #1\vspace{1pc}}}

\newtheorem{theorem}{Theorem}

\newtheorem{teo}{Theorem}

\newtheorem{corollary}[teo]{Corollary}
\newtheorem{lemma}[teo]{Lemma}

\theoremstyle{definition}

\theoremstyle{remark}

\title[]{On Liouville's theorem and the Strong Liouville Property}
\author{John E. Bravo and Jean C. Cortissoz }
\email{jcortiss@uniandes.edu.co, j.bravob@uniandes.edu.co}
\address{Department of Mathematics, Universidad de los Andes, Bogot\'a DC, Colombia}
\date{}

\begin{document}

\begin{abstract}
We explore Liouville’s theorem and the Strong Liouville Property (SLP) for harmonic functions on Riemannian cones and surfaces. Our approach recasts the classical Liouville property in terms of the growth of radial eigenfunctions (in the case of manifolds with rotational symmetry), allowing us to recover and sharpen known results under minimal assumptions. We provide explicit estimates for the slowest-growing nonconstant harmonic functions on cones and surfaces, and construct examples where doubling fails but Liouville and SLP still hold. Finally, we prove
a nonlinear Liouville theorem for $p$-subharmonic functions,
$p\geq 2$, under curvature bounds,
in complete Riemannian
surfaces with a pole which simultaneously recover Milnor’s and Cheng–Yau’s theorems as particular cases. This result appears to be new and suggests a unified geometric perspective on linear and nonlinear Liouville phenomena.
\end{abstract}

\maketitle

\keywordsnames{Liouville's Theorem; bounded harmonic functions; Strong Liouville Property; $p$-Laplacian; radial eigenfunction; curvature estimates}

{\amsprimary {31C05, 53C21}}

\tableofcontents

\section{Introduction}

One of the most celebrated results in Complex Analysis is Liouville's theorem. It states that
any bounded entire function $f:\mathbb{C}\longrightarrow \mathbb{C}$ must be constant. 
Liouville's theorem has spawned a lot of research along the past two centuries, and different 
versions have been developed in geometric contexts and have inspired important
and beautiful research in Geometric Analysis.

\medskip
In this paper we study Liouville's theorem and the Strong Liouville Property
(to be
defined below) in a geometric context, and to be more precise, on Riemannian cones. Our goal is to test the limits of 
some known, and one might say, celebrated results on the behavior of
harmonic functions on Riemannian manifolds, such as Cheng-Yau Liouville's theorem
and Ni-Tam results on surfaces of finite total curvature, and to extend the study 
of this result to nonlinear extensions of the Laplacian, e.g., to the
$p$-Laplacian on surfaces with a pole.

\medskip
To begin with our discussion, we give a few definitions.
First of all,
given a closed $\left(n-1\right)$-dimensional Riemannian manifold $\left(N,g_{\omega}\right)$, we define the topological cone
as the (topological) quotient
\[
M=\left(\left[0,\infty\right)\times N\right)/\left(\left\{0\right\}\times N\right).
\]
The equivalence class of $\left\{0\right\}\times N$ is called the vertex of the cone.

\medskip
Let $g$ be a metric defined on the cone $M$ by
\[
g=dr^2+\phi\left(r\right)^2g_{\omega},
\]
where $\phi$ is a smooth function such that $\phi\left(0\right)=0$ and $\phi'\left(0\right)=1$ and
suitable behavior at $r=0$ (for regularity issues). From now on, we shall call $\left(M,g\right)$ a 
\emph{Riemannian cone}.

\medskip
For $r>0$, the Laplace operator in the Riemannian cone is defined as
\[
\Delta=\frac{\partial^2}{\partial r^2}+\left(n-1\right)\frac{\phi'}{\phi}\frac{\partial}{\partial r}+\frac{1}{\phi^2}\Delta_{\omega},
\]
where $\Delta_{\omega}$ is the Laplacian on $N$.
We say that $u\in C\left(M\right)\cap C^2\left(\left(0,\infty\right)\times N\right)$
is harmonic if $\Delta u=0$ on $\left(0,\infty\right)\times N$.

\medskip
Liouville's theorem on $\mathbb{R}^n$ (which corresponds to the
case $\phi\left(r\right)=r$ and $N=\mathbb{S}^{n-1}$ with
the round metric of curvature 1) asserts that a bounded harmonic function $u:\mathbb{R}^n\longrightarrow \mathbb{R}$
is constant. So, we say that a Riemannian cone satisfies Liouville's theorem
if any bounded harmonic function must be constant. This, naturally, raises the question
of what is the minimal growth a non-constant harmonic function must have; for instance, in $\mathbb{R}^n$
with the Euclidean metric the answer is ``at least linear".

\medskip
Besides Liouville's theorem, as we said above, we want to discuss a property called Strong Liouville Property
(abbreviated SLP). 
On a Riemannian cone, we define, for $d$ a positive integer, the spaces 
\[
\mathcal{H}_d\colon = \left\{u:M\longrightarrow \mathbb{R} \,\,
\mbox{harmonic}:\, 
\sup_{\left|x\right|}\frac{\left|u\left(x\right)\right|}{\left(1+\left|x\right|\right)^d}<\infty \right\},
\]
where we denote by $\left|x\right|$ the distance to the vertex of the cone.
A Riemannian cone is said
to satisfy the Strong Liouville Property if $\mbox{dim}\left(\mathcal{H}_d\right)<\infty$
for all $d$.

\medskip
Before we state our first result (from which applications follow), we need a little more
preparation. On a cone, we can use separation of variables to find harmonic
functions. Indeed, if 
$f_{m,k}$ is an eigenfunction for the $k$-th eigenvalue of the
Laplacian $\Delta_{\omega}$ of $N$, it is easy to prove that
we can obtain a harmonic function by writing
\[
h_{m,k}\left(r,\omega\right)=\varphi_m\left(r\right)f_{m,k}\left(\omega\right),
\]
where $\varphi_m\in C^2\left(\left(0,\infty\right)\right)\cap C\left(\left[0,\infty\right)\right)$ 
which satisfies
\begin{equation}
\label{eq:radial_equation}
\varphi_m''+\left(n-1\right)\frac{\phi'}{\phi}\varphi_m'-\frac{\lambda_m^2}{\phi^2}\varphi_m=0.
\end{equation}
We are now ready to state:
\begin{theorem}
\label{thm:abstract_liouville}
 Let $u:M\longrightarrow \mathbb{R}$ be a harmonic function, and let
\[
M_R=\left(\left[0,R\right)\times N\right)/\left(\left\{0\right\}\times N\right).
\]
and define
 \[
 u_R= \max_{x\in \partial M_R}\left|u\left(x\right)\right|.
 \]
 Assume there is an $A\left(R\right)\rightarrow \infty$ such that $\varphi_m\left(R\right)\gtrsim A\left(R\right)$,
 $m\geq 1$. If $u_R=o\left(A\left(R\right)\right)$, then $u$ is constant.
\end{theorem}

As a consequence of this result, we obtain the following.
\begin{corollary}
\label{corollary:boundedphi}
 We are under the assumptions of Theorem \ref{thm:abstract_liouville}. 
 If $\phi'\left(r\right)\leq \beta$,
 then Liouville's theorem holds on $M$. Furthermore, an explicit estimate on the growth of the slowest growing non-constant harmonic function on $M$ can be given in terms of $\beta$ and of the smallest eigenvalue of the Laplacian on $N$. In fact, in the 
 case when $N=\mathbb{S}^{n-1}$ with the round metric, 
 if
 \[
 n\geq 2\left(1+\frac{1}{\beta^2}\right)+2\sqrt{\left(1+\frac{1}{\beta^2}\right)^2-\left(1+\frac{1}{\beta^2}\right)},
 \]
 we have that
 the minimal required growth of a non-constant harmonic function on $M$ must be
 at least greater than
 $
 r^{\frac{n-1}{\sqrt{2}\beta^2 \left(n-2\right)}}
 $ (this estimate is not sharp).
\end{corollary}
Notice that the
corollary applies to the case of
an open manifold of nonnegative radial
curvature (which is given by $-\phi''/\phi$)
and of finite total radial curvature. In the case of nonnegative radial curvature
the estimate on $\phi$ is satisfied for $\beta=1$, and then our result is valid for
$n\geq 7$. However, observe that if the condition on $\phi$ is satisfied for 
a given $\beta$, then it is also satisfied for any $\beta'\geq \beta$, 
so by picking an appropriate large enough $\beta'$, it might be shown that our result holds for any $n\geq 3$, 
with estimates on the growth of the slowest growing harmonic function depending now on the new appropriately chosen $\beta'$. 
As special as it is, we shall treat the two-dimensional case on its own.

\medskip
Also, regarding the previous corollary, Saloff-Coste in \cite{Saloff-Coste92}, proved that Liouville's theorem
holds if the manifold satisfies certain geometric-analytic conditions, among them, that the manifold satisfies
a doubling condition, that is for any $R>0$ there is a constant $\kappa$
independent of $R$ such that for any two concentric balls $B_R$ and $B_{2R}$
\[
\frac{\mbox{Vol}\left(B_{2R}\right)}{\mbox{Vol}\left(B_{R}\right)}\leq \kappa.
\]
Using the explicit representation for harmonic functions in a rotationally
symmetric open surface, in Section \ref{subsect:counterexample}, we will show
an example that violates the previous assumption but for which Liouville's theorem
and the SLP hold.

\medskip
Before going any further, the reader should notice that in the case where the cone is $\mathbb{R}^2$ with the Euclidean metric, it is easy to see that $A\left(R\right)=R$ 
is a good choice. Indeed,
for the case of $\mathbb{R}^2$ it is well known that $\phi\left(r\right)=r$, that $\lambda_m^2=m$, and thus
that we can take $\varphi_m\left(r\right)=r^m$ in equation (\ref{eq:radial_equation}).
This implies the sublinear version of Liouville's
theorem (and the reader
can see how it follows from Theorem \ref{thm:abstract_liouville}): Every harmonic function defined on $\mathbb{R}^2$ of sublinear
growth must be constant.

\medskip
Our second result is the following.
\begin{theorem}
Let $\left(M,g\right)$ be an $n$-dimensional Riemannian cone with metric 
\[
g=dr^2+\phi\left(r\right)^2g_{N},
\]
and $\phi'\left(r\right) \leq \beta$. 
Then $M$ satisfies the SLP.
\end{theorem}

\medskip
Related to this result, various authors have contributed to give an answer to 
the
question on how fast should grow a non-constant harmonic function, some of them
in more general context (we cite a few: \cite{Co, Grygoryan92, Kasue, Saloff-Coste92}). 
However, the case considered in
this paper provide more explicit estimates and examples that might lead to
some refinements on known results. Regarding surfaces, our studies include 
surfaces with no finite total curvature (see the example at the end of Section
\ref{subsect:boundedphi}; this should be contrasted
with the results in \cite{LT91}).

\subsection{A nonlinear Liouville theorem}
In the final section of the paper, we extend the approach
from \cite{Co} to the nonlinear setting by proving a Liouville-type theorem for \(p\)-subharmonic functions on complete surfaces with a pole. We show for $p\geq 2$ that if the Gaussian curvature satisfies either
\[
K \geq -\frac{p - 1}{r^2 \log r} \quad \text{or} \quad K \geq -\frac{(p - 1)(p - 2)}{r^2}
\]
outside a compact set, then any \(p\)-subharmonic function with sufficiently slow growth is constant. This result generalizes both Milnor’s theorem for harmonic functions and the Cheng--Yau Liouville theorem, and appears to be new in the literature. It provides a sharp curvature threshold for constancy of \(p\)-subharmonic functions, at least for \(p \geq 2\).

\medskip
\subsection{Organization of the paper.}
This paper is organized as follows. In Section 
\ref{sect:proof}
we prove Theorem \ref{thm:abstract_liouville}, in Section \ref{sect:applications} we present 
some applications, among them a proof
of Corollary \ref{corollary:boundedphi}, in Section \ref{sect:slp} we discuss the Strong Liouville Property,
and in Section \ref{sect:nonlinear_liouville} we present
our nonlinear generalization of Liouville's
theorem.



\section{Proof of Theorem \ref{thm:abstract_liouville}}
\label{sect:proof}

Let $u:\mathbb{R}^n\longrightarrow \mathbb{R}$ be a harmonic function.
For $R>0$ arbitrary, recall that
\[
M_R=\left(\left[0,R\right)\times N\right)/\left(\left\{0\right\}\times N\right),
\]
and 
\[
u_R\left(\omega\right)=u|_{\partial M_R}.
\]
Since $u_R$ is smooth, we can expand it as the Fourier series
\[
u_R\left(\omega\right)=\sum_k\sum_m c_{k,m,R}f_{m,k}\left(\omega\right),
\]
where $f_{m,k}$ is an eigenfunction of the eigenvalue $\lambda_m^2$
of the Laplacian $\Delta_{\omega}$ of $N$. From now on, we shall
assume that the $\left\{f_{m,k}\right\}$ form an orthonormal basis
for $L^{2}\left(N\right)$.

\medskip
The harmonic extension of $u_R$ to $M_R$, and thus, by uniqueness,
$u$ on $M_R$
(see Theorem 3 in \cite{Co2}), is given by
\begin{equation}
\label{eq:radius_R}
u\left(r,\omega\right)=\sum_m \frac{\varphi_m\left(r\right)}{\varphi_m\left(R\right)}
\sum_k c_{k,m,R}f_{m,k}\left(\omega\right).
\end{equation}
Indeed, the proof is exactly the
same as that given in \cite{Co2} for the Dirichlet Problem at Infinity.
First, we must show that the sum converges in $L^2$-sense for
every $r\leq R$. Here we need the fact that
the $\varphi_m$'s can be chosen to be nonnegative and \emph{nondecreasing}:
that this can be done is shown in \cite{Co2}. 
Therefore $0\leq\varphi_m\left(r\right)/\varphi_m\left(R\right)\leq 1$ for 
$0\leq r\le R$, and, from having convergence at $R$, our assertion follows
immediately. 
Next, we must show 
that the boundary condition at $\partial M_R$ is met as $r\rightarrow R$, and for this we also ask the
reader to consult \cite{Co2}. 
From now on, we shall assume, without loss of generality, that $c_{0,0}=0$. 

\medskip
Thus far, we have shown that in $M_R$, $u$ can be written as. 
We shall prove that $u$
can be represented, in the whole of $M$, as
\begin{equation}
\label{eq:radius_1}
\sum_k\frac{\varphi_m\left(r\right)}{\varphi_m\left(1\right)}\sum_m c_{m,k,1}f_{m,k}\left(\omega\right),
\end{equation}
where 
\[
\sum_{k,m}c_{m,k,1}f_{m,k}\left(\omega\right)
\]
is the Fourier expansion of $u$ restricted to $\partial M_1$.
By evaluating (\ref{eq:radius_R}) at $R=1$ we obtain (\ref{eq:radius_1}).

\medskip
Hence, by considering (\ref{eq:radius_R}) at any $R>1$ and evaluating at $r=1$
and comparing with (\ref{eq:radius_1}) when evaluated at $r=1$,
by uniqueness, we must have that
\[
c_{m,k,R}=\frac{\varphi_m\left(R\right)}{\varphi_m\left(1\right)}c_{m,k,1},
\]
but then, it follows that (\ref{eq:radius_1}) represents $u$ for any $R\geq 1$, and
as this identity is already valid for $R<1$, our claim is proved.

\medskip
Now we can use Plancherel's theorem to compute
\[
\left\|u\right\|^2_{L^2\left(\partial M_R\left(O\right)\right)}
=
\mbox{Vol}\left(N,g_{\omega}\right)
\sum_k\left(\frac{\varphi_m\left(R\right)}{\varphi_m\left(1\right)}\right)^2\sum_m \left|c_{m,k,1}\right|^2\phi^{n-1}\left(R\right),
\]
but if we assume that $u\left(\omega, R\right)=o\left(A\left(R\right)\right)$,
we must have that
\[
\sum_k\left(\frac{\varphi_m\left(R\right)}{\varphi_m\left(1\right)}\right)^2\sum_m \left|c_{m,k,1}\right|^2\phi^{n-1}\left(R\right)
= o\left(A\left(R\right)^2\right)\phi^{n-1}\left(R\right),
\]
which, under the hypotheses of the theorem,
that is because of the fact that $\varphi_m\left(R\right)\gg A\left(R\right)$, $m\geq 1$, it
must hold that $c_{m,k,1}=0$ for all $m,k$, except possibly for
$m=k=0$, that is, $u$ must be constant.

\medskip
The argument above has some bearing in the discussion of the Strong Liouville Property (SLP) for Riemannian cones, as we shall see in Section \ref{sect:slp} below.

\section{Applications}
\label{sect:applications}

\subsection{The case of dimension 2}

Before going to the applications of Theorem \ref{thm:abstract_liouville} as stated in the corollary, we discuss the case
of dimension 2, which is not included in any of them. 
 That is, we discuss the case of $\mathbb{R}^2$ with a metric of the
form
\[
dr^2+\phi\left(r\right)^2 d\theta^2.
\]
More general cases are discussed in \cite{Co}, but we want to recall a few facts
and place them in this new framework (also, the reader should put this discussion in the
context of Milnor's celebrated paper \cite{Milnor}).

\medskip
First of all, the
representation (\ref{eq:radius_1}) is valid for any harmonic function.
Thus, to get the slowest possible non-constant harmonic function
on $\mathbb{R}^2$ with the given metric,
we need to bound $\varphi_m$, and
in this case, the radial equation is given by
\[
\varphi_m''+\frac{\phi'}{\phi}\varphi_m'-\frac{m^2}{\phi^2}\varphi_m =0,
\]
and it can be explicitly solved, and so we have
\[
\varphi_m\sim \exp\left(m \int_1^r\frac{1}{\phi\left(s\right)}\,ds\right).
\]
Therefore, we can take as $A$ the function $\exp\left(\int_1^r\frac{1}{\phi\left(s\right)}\,ds\right)$,
and if the Gaussian curvature satisfies $K\geq 0$, then $\phi\left(r\right)\leq r$, and hence
in this case we recover Liouville's theorem in the case of nonnegative Gaussian curvature:
a harmonic function of sublinear growth must be constant.

\medskip
On the other hand, notice that if $\phi\left(r\right)\sim r^{\beta}$ with $0<\beta<1$, we see that 
there cannot be non-constant harmonic functions of polynomial growth, and, further, we also get 
that the slowest growing non-constant harmonic function 
grows at least as
\[
\exp\left(\frac{1}{1+\beta}r^{1-\beta}\right).
\]
The previous analysis is related to the following fact. Being harmonic
is a conformal invariant in dimension 2. Since the functions
of the form $r^me^{\pm m i\theta}$ generate the space of harmonic functions
in $\mathbb{R}^2$, one should expect that 
$\exp\left(m\int_1^{r}\frac{1}{\phi\left(s\right)}\,ds\right)e^{\pm i m\theta}$
generates it on a surface with a rotationally symmetric metric
(see also \cite{Bravo}).

\subsection{The case of $\phi' \leq \beta $}
\label{subsect:boundedphi}

This case includes that of nonnegative radial curvature and also that
of finite total radial curvature. So,
let us assume that $\phi'\leq \beta$ (recall that in any case we are assuming 
that $\phi'\geq 0$).
To apply the previous theorem, we need to find what works as $A\left(r\right)$
in this case.
To do so, let us write
\[
\varphi_m\left(s\right) = B\exp\left(\int_1^s \frac{\lambda_m^2}{\phi^{n-1}\left(\tau\right)}x_m\left(\tau\right)\, d\tau\right),
\]
with $B=\varphi_m\left(1\right)$,
for $s\geq 1$, and $x_m\left(t\right)$ a smooth function. 
As shown in \cite{Co2}, $x_m\left(t\right)$ satisfies the Riccati equation
\begin{equation}
\label{eq:riccati}
x_m'\left(s\right)+\frac{\lambda_m^2}{\phi^{n-1}\left(s\right)}x_m^2\left(s\right)=\phi^{n-3}\left(s\right).
\end{equation}
Since the $\varphi_m$'s are non-decreasing (by the arguments in \cite{Co2}), it follows that $x_m'\left(t\right)\geq 0$
which is a basic but important observation.

\medskip
Let 
\[
\eta_m=\frac{1}{2}\left(-\frac{\beta\left(n-2\right)}{\lambda_m}+\sqrt{\frac{\beta^2\left(n-2\right)^2}{\lambda_m^2}+4}\right).
\]
At any given time, either
\[
\frac{\lambda_m^2}{\phi^{n-1}\left(s\right)}x_m^2\left(s\right)> \eta_m^2\phi^{n-3}\left(s\right),
\]
that is
\begin{equation}
\label{ineq:ineq_0}
x_m\left(s\right)\geq \frac{\eta_m}{\lambda_m}\phi^{n-2},
\end{equation}
or
\begin{equation}
\label{ineq:ineq_1}
x_m'\left(s\right)> \left(1-\eta_m^2\right) \phi^{n-3}\left(s\right).
\end{equation}
Our choice of $\eta$ guarantees that
\[
\frac{1-\eta_m^2}{\beta\left(n-2\right)}=\frac{\eta_m}{\lambda_m}.
\]

We shall show the following. Assume that at time $t=t_0$ the 
following estimate holds.
\begin{equation}
\label{ineq:ineq_3}
x_m\left(t\right)\geq \frac{1-\eta_m^2}{\beta\left(n-2\right)}\phi^{n-2}\left(t\right),
\end{equation}
then there is an $\epsilon>0$ such that the same estimate is valid
on $\left(t_0,t_0+\epsilon\right)$. To do so, choose $\epsilon>0$ such that
either (\ref{ineq:ineq_0}) or (\ref{ineq:ineq_1}) is valid. In the first case, (\ref{ineq:ineq_3}) follows immediately.
In the second case, we estimate as follows. Let $\tau\in \left(t_0,t_0+\epsilon\right)$, then
\begin{eqnarray*}
 x_m\left(\tau\right)&=&x_m\left(t_0\right)+\int_{t_0}^{\tau}\phi^{n-3}\left(s\right)\,ds\\
 &\geq& x_m\left(t_0\right)+
 \left(1-\eta_m^2\right)\int_{t_0}^{\tau}\phi^{n-3}\left(s\right)\frac{\phi'\left(s\right)}{\beta}\,ds\\
 &=& x_m\left(t_0\right)+\frac{1-\eta_m^2}{\beta\left(n-2\right)}\left(\phi^{n-2}\left(\tau\right)
 -\phi^{n-2}\left(t_0\right)\right)\\
 &\geq& 
 \frac{1-\eta_m^2}{\beta\left(n-2\right)}\phi^{n-2}\left(\tau\right).
\end{eqnarray*}
Notice that by continuity, estimate (\ref{ineq:ineq_3}) holds on a closed set, and that 
this set is nonempty, 
since the estimate is satisfied when $t=0$. Therefore, the estimate holds 
for all $t\in \left[0,\infty\right)$. Hence we have 
\[
x_m\left(t\right)\geq \frac{\eta_m}{\lambda_m}\phi^{n-2}\left(t\right).
\]
Thus, we can take
\[
\eta_m=\frac{1}{2}\frac{\lambda_m^2}{\beta\left(n-2\right)}\left[\sqrt{1+\frac{4\lambda_m^2}{\beta^2\left(n-2\right)^2}}-1\right],
\]
and this gives
\begin{equation}
\label{eq:harmonic_growth}
x_m\left(t\right)\geq 
a_m
\phi^{n-2}\left(t\right),
\end{equation}
where 
\[
a_m=\frac{1}{2}\frac{\lambda_m}{\beta\left(n-2\right)}\left[\sqrt{1+\frac{4\lambda_m^2}{\beta^2\left(n-2\right)^2}}-1\right].
\]
This shows that 
\[
\varphi_m \geq \exp\left(a_m\int_1^t\frac{1}{\phi\left(s\right)}\,ds\right),
\]
and, noticing that the $a_m$'s form an increasing sequence, if a harmonic function
is $o\left(A\left(R\right)\right)$ with
\[
A\left(t\right)= \exp\left(a_1\int_1^t\frac{1}{\phi\left(s\right)}\,ds\right),
\]
then it must be constant. Therefore, Liouville's theorem will hold as soon as
$\displaystyle\int_1^{\infty} \frac{1}{\phi\left(s\right)}\,ds =\infty$. But if $\phi'\left(r\right)\leq \beta$,
then $\phi\leq \beta r$, and in consequence Liouville's theorem holds.

\medskip
We now examine what this says about Liouville's theorem for rotationally symmetric metrics
in $\mathbb{R}^n$. In this case,
we have the round metric on $N=\mathbb{S}^{n-1}$, and it is well known that
$\lambda_1^2=n-1$. Thus, we can take
\[
\eta_1=\frac{1}{2}\frac{\beta\left(n-2\right)}{\sqrt{n-1}}\left[
\sqrt{1+\frac{4}{\beta^2}\frac{n-1}{\left(n-2\right)^2}}-1\right].
\]
Just to show a simple estimate (without being sharp), notice
that
\[
\sqrt{1+x}-1\geq \frac{1}{2\sqrt{2}}x,\quad 0\leq x\leq 1.
\]
Therefore, if
\[
n\geq 2\left(1+\frac{1}{\beta^2}\right)+2\sqrt{\left(1+\frac{1}{\beta^2}\right)^2-\left(1+\frac{1}{\beta^2}\right)},
\]
then we can take
\[
\eta_1 \geq \frac{1}{\beta\sqrt{2}}\frac{\sqrt{n-1}}{n-2},
\]
and hence, since it is not difficult to see that $\eta_m\geq \eta_1$, this yields
\[
x_m\left(t\right)\geq \frac{1}{\beta\sqrt{2}\left(n-2\right)}\phi^{n-2}\left(t\right).
\]
We have obtained an estimate:
\begin{equation}
\label{eq:harmonic_growth_2}
\varphi_m\left(t\right)\geq \exp\left(\frac{n-1}{\beta\sqrt{2}\left(n-2\right)}\int_1^t\frac{1}{\phi\left(s\right)}\,ds\right)
\geq t^{\frac{n-1}{\beta^2 \sqrt{2} \left(n-2\right)}},
\end{equation}
using the assumption that $\phi\left(r\right)\leq \beta r$. This proves
our estimate for the slowest growing harmonic function in the case considered, and
provides a proof of Corollary \ref{corollary:boundedphi}.

\medskip
When the radial curvature is nonnegative we have that $0\leq \phi\left(r\right)\leq r$, 
so we can take $\beta=1$ and recover Liouville's theorem
for bounded harmonic functions. 
Furthermore, if $\phi\left(r\right)=o\left(r\right)$ from estimate (\ref{eq:harmonic_growth})
it can be deduced that harmonic functions of at most polynomial growth are constant.

\subsubsection{Examples} To see the reach of the previous estimates, consider the following example:
\[
\phi = \frac{1}{2}\left(r+\sin r\right).
\]
In this case $\beta=1$, and thus Liouville's theorem holds. Moreover, any non-constant
harmonic function must grow faster that $r^{\frac{1}{\sqrt{2}}\frac{n-1}{n-2}}\sim r^{0.7071\dots}$ 
for $n$ large. Notice that
here Cheng-Yau does not apply, since the radial curvature in this case
is given by
\[
K\left(r\right)=\frac{\sin r}{r+\sin r}.
\]
From Gauss equation, it is not difficult to see that this term dominates
the expression for the radial Ricci tensor,
so the radial part of the Ricci tensor is sometimes
negative (even outside any compact set).

\medskip
In the case of surfaces,
notice that a non-constant harmonic
function must be at least of quadratic
growth, because in this case the slowest possible
growth for a harmonic function is given by
(see \cite{Bravo})
\[
\exp\left(\int^r \frac{2}{s+\sin s}\,ds\right)\sim r^2.
\]
Compare this estimate with Theorem 4.3 in \cite{Saloff-Coste92}.
Notice also that a surface with the metric described above does not have finite total 
radial curvature.
Indeed, $K=\sin r/\left(r+\sin r\right)$, and hence
\[
\int_M\left|K\right|\,dA=\frac{1}{2}\int_0^{\infty}\left|\sin r\right|\,dr.
\]
However, it does satisfy the Strong Liouville Property as well. The reader should also
compare this with the work developed in \cite{LT91}, where the case of surfaces of
finite total curvature is covered. 

Of course, the previous observations raise the question of what would be 
a sharp assumption on the curvature of an open Riemannian surface $M^2$ so that it satisfies the
SLP.

\section{Strong Liouville Property}
\label{sect:slp}

First of all, observe that the fact that we can expand
any harmonic function as
(\ref{eq:radius_1}) and an application of Plancherel's theorem, just as is done in
(\ref{sect:proof}), shows that
the space of harmonic functions whose growth is slower than $r^d$, is spanned 
by the set
\[
\left\{\varphi_m\left(r\right)f_{m,k}\left(\omega\right):\, \varphi_m\left(r\right)\lesssim r^d\right\}.
\]

\medskip
Under the assumption that 
$\phi'\left(r\right)\leq \beta$,
the dimension of the space of harmonic functions of growth $r^d$ is finite-dimensional. 

Next, observe that, since $\lambda_m\rightarrow \infty$ as $m\rightarrow \infty$, 
from estimate (\ref{eq:harmonic_growth}) we obtain for $m$ large enough
\[
x_m\left(t\right)\geq 
\frac{1}{2}\frac{\lambda_m}{\beta\left(n-2\right)},
\]
and thus
\[
\varphi_m\left(r\right)\gtrsim r^{\frac{\lambda_m^3}{\beta^2\left(n-2\right)}},
\]
for $m$ large enough, which immediately shows that SLP holds.

Here, it is interesting that we can show the following family of examples, where 
any bounded harmonic function must be constant, but the strong Liouville property does not hold:
in particular $\mbox{dim}\left(\mathcal{H}_1\right)=\infty$.
Take $\mathbb{R}^n$ with a metric of the form
\[
g=dr^2+\phi\left(r\right)^2g_{\mathbb{S}^{n-1}},
\]
with $\phi\left(r\right)\sim r\log r$ for $r$ large, with $\phi'\geq 0$. We give a short argument for this.
From the Riccati equation (\ref{eq:riccati}), and the fact that $x_m'\left(t\right)\geq 0$ we can deduce that
\[
x_m\left(t\right)\leq \frac{1}{\lambda_m}\phi^{n-2}\left(t\right),
\]
so that 
\[
\varphi_m\left(r\right)\lesssim \exp\left(\lambda_m \int^r\frac{1}{\phi\left(s\right)}\,ds\right).
\]
Therefore, 
\[
\varphi_m\lesssim\left(\log r\right)^{m},
\]
and our assertion follows.

\subsection{On the doubling condition: An example in dimension 2}
\label{subsect:counterexample}
Here we show an example of a metric in $\mathbb{R}^2$ for which 
the doubling condition, as described in the introduction does not hold, 
but for which both, Liouville's theorem and the SLP hold.

\medskip
Let $\psi:\mathbb{R}\longrightarrow \left[0,1\right]$ be a standard bump function such that
$\psi\equiv 1$ on $\left[-1,1\right]$, $\mbox{supp}\left(\psi\right)\subset \left[-2,2\right]$,
increasing in $\left[-2,-1\right)$, decreasing
in $\left(1,2\right]$, with
$\psi\left(-\frac{3}{2}\right)=\frac{1}{2}
=\psi\left(\frac{3}{2}\right)$, and symmetric
with respect to $r=0$.
Define
\begin{equation}
\label{eq:counterexample}
\phi\left(r\right)=\sum_{k=0}^{\infty}r\psi\left(r-6k\right) + \sum_{k=0}^{\infty} e^r \psi\left(r-\left(6k+3\right)\right).
\end{equation}
It is not difficult to show that there is a constant such that
\[
c\log r \leq \int_1^r \frac{1}{\phi\left(s\right)}\,ds \leq \log r.
\]
These two inequalities show that $\mathbb{R}^2$ with the metric
\[
g=dr^2+\phi\left(r\right)^2 d\theta^2,
\]
satisfies Liouville's theorem and the SLP. 

However, the exponential growth in volume provided by the second sum on the 
right-hand side of (\ref{eq:counterexample}) does not allow for the doubling property to hold. In fact,
\[
\mbox{Vol}\left(B_{6k}\right)\leq
e^{6k},
\]
and 
\[
\mbox{Vol}\left(B_{12k}\right)
\geq \mbox{Vol}\left(B_{6\left(2k-1\right)}\right)
\geq \int_{12k-7}^{12k-6}e^r\,dr=e^{12k-7}
\left(e-1\right).
\]
Therefore,
\[
\frac{\mbox{Vol}\left(B_{12k}\right)}{\mbox{Vol}\left(B_{6k}\right)}\geq 
e^{6k-7}\left(e-1\right)
\rightarrow \infty.
\]
This shows that, although sufficient, the doubling
condition is not necessary for Liouville's theorem
or the SLP to hold.

\section{A nonlinear Liouville's theorem: the case of the
$p$-Laplacian}
\label{sect:nonlinear_liouville}

Having established Liouville-type results for harmonic functions, we now explore their nonlinear analogs for the 
$p$-Laplacian, which generalize classical theorems under sharper curvature conditions.

\medskip
In this section, we will consider $\mathbb{R}^2$ with a 
metric of the form
\begin{equation}
\label{eq:metric_general_form}
g=dr^2 + \phi\left(\theta, r\right)^2d\theta^2.
\end{equation}

The $p$-Laplacian of a function $u$ is given by
\[
\Delta_p u = \mbox{div}
\left(\left|\nabla u\right|^{p-2}\nabla u\right).
\]
The $p$-Laplacian comes
from the Euler-Lagrange equations satisfied
by the functionals
\[
\int_M \left|\nabla u\right|^{p}\, dV_g, \quad p>1.
\]
The case $p=2$ corresponds
to the usual Laplacian, and
hence, the $p$-Laplacian 
appears as a natural, nonlinear, generalization of the
Laplacian.

\medskip
Although most classical Liouville-type theorems are for the Laplacian, here we extend the method used in
\cite{Co}
to the nonlinear setting of the $p$-Laplacian.
We begin with the following
general comparison principle. In order to state 
our basic comparison result, we introduce the
notation
\[
M\left(u,r\right)=\sup_{\theta\in \mathbb{S}^1}
\left|u\left(\theta,r\right)\right|.
\]

\begin{theorem}
\label{thm:comparison_1}
 Let 
 $u:\mathbb{R}^2\longrightarrow\mathbb{R}$, with
 $\mathbb{R}^2$ endowed with a smooth complete Riemannian
 metric given in polar coordinates by
 (\ref{eq:metric_general_form}).
 Assume that its Gaussian curvature $K$ satisfies
 \[
 K\left(\theta,r\right)\geq 
 -\frac{z''\left(r\right)}{z\left(r\right)},
 \]
 outside a compact subset.
 
 Let 
 \[
 h\left(r\right)=\int^r
 \frac{1}{z\left(s\right)^{\frac{1}{p-1}}}\,ds.
 \]
 If $u\in C^1\left(\mathbb{R}^2\right)$ 
 is $p$-subharmonic ($\Delta_p u\geq 0$ in the weak sense) and 
 \[
 \liminf_{r\rightarrow \infty}
 \frac{M\left(u,r\right)}{h\left(r\right)}=0,
 \]
 then $u$ must be constant.
\end{theorem}
For the proof, we shall need the following computation.

\begin{lemma}
Let $g$ be a smooth Riemannian metric on $\mathbb{R}^2$ of 
the form 
(\ref{eq:metric_general_form}), where $\phi$ is not necessarily radial. Let $u = u(r)$ be a radial function. Then the $p$-Laplacian of $u$ satisfies
\[
\Delta_p u(r) = |u'(r)|^{p-2} \left[ (p - 1) u''(r) + \frac{\partial_r \phi(r, \theta)}{\phi(r, \theta)} u'(r) \right].
\]
\end{lemma}

\begin{proof}
Since $u = u(r)$, we have $\nabla u = u'(r) \partial_r$, so $|\nabla u| = |u'(r)|$, and
\[
\Delta_p u = \operatorname{div}_g \left( |u'(r)|^{p-2} u'(r) \partial_r \right).
\]
We write $X = |u'(r)|^{p-2} u'(r) \partial_r$. Then, in polar coordinates, the divergence of a radial vector field $X^r \partial_r$ is
\[
\operatorname{div}_g X = \frac{1}{\phi(r,\theta)} \frac{\partial}{\partial r} \left( \phi(r,\theta) X^r \right).
\]
Therefore,
\[
\Delta_p u = \frac{1}{\phi(r,\theta)} \frac{d}{dr} \left( \phi(r,\theta) |u'(r)|^{p-2} u'(r) \right).
\]
Expanding using the product rule:
\[
\Delta_p u = \frac{1}{\phi(r,\theta)} \left[ \partial_r \phi(r,\theta) \cdot |u'(r)|^{p-2} u'(r) + \phi(r,\theta) (p-1) |u'(r)|^{p-2} u''(r) \right].
\]
Finally, factor out $|u'(r)|^{p-2}$:
\[
\Delta_p u = |u'(r)|^{p-2} \left[ (p-1) u''(r) + \frac{\partial_r \phi(r,\theta)}{\phi(r,\theta)} u'(r) \right].
\]
\end{proof}

We are ready to 
give a proof of Theorem \ref{thm:comparison_1}.
From the lemma, it is easy to prove that if 
for $r\geq a$
\[
K\left(\theta, r\right)\geq
-\frac{z''\left(r\right)}{z\left(r\right)},
\]
and
\[
\frac{z'\left(a\right)}{z\left(a\right)} \geq \max_{\theta\in \mathbb{S}^1} 
\frac{\phi_r\left(\theta,a\right)}{\phi\left(\theta,a\right)},
\]
then $h$ is a $p$-supersolution. Indeed, by the
Sturm comparison theorem (see \cite{Co})
\[
\frac{\phi_r\left(\theta,r\right)}{\phi\left(\theta,r\right)}\leq
\frac{z'\left(r\right)}{z\left(r\right)},\quad r\geq a,
\]
and hence
\[
0=(p-1) h''(r) + \frac{z'\left(r\right)\phi(r)}{z(r)} h'(r)
\geq (p-1) h''(r) + \frac{\partial_r \phi(r,\theta)}{\phi(r,\theta)} h'(r),
\]
which shows that
\[
\Delta_p h \leq 0.
\]
We are ready to give a proof of Theorem 3. Let 
$R_1>0$ and $\delta>0$. Consider the function
\[
v=\delta h + M\left(u,R_1\right).
\]
Clearly $\Delta_p v\leq 0$. On the other hand,
by the hypothesis, there is $R_2>R_1$, which depends 
on $\delta$, such that
\[
M\left(u,R_2\right)< v\left(R_2\right),
\]
and hence, by the comparison principle for the
$p$-Laplacian,
in the annulus $R_1\leq r\leq R_2$,
$u\leq v$. Pick any $R>R_1$; the previous estimates
imply, by taking $\delta>0$ small enough, that
the maximum of $u$ in $B_R$ is reached at
an interior point: therefore, by the
strong maximum principle for
$p$-subharmonic functions
(see Theorem 4.2 in \cite{GoffiPedicone2021}), $u$ must be constant.

\medskip
We apply Theorem \ref{thm:comparison_1} to show the following.
\begin{theorem}
Let $\mathbb{R}^2$ endowed with a complete metric $g$ of the
form (\ref{eq:metric_general_form}).
Let $p\geq 2$, and let 
 assume that outside a compact subset
 the Gaussian curvature $K$ of $g$ satisfies 
 either
 \begin{itemize}
 \item[(i)]
 \[
 K\geq -\frac{p-1}{r^2\log r},
 \]
 \item[(ii)] or
 \[
 K\geq -\frac{\left(p-1\right)\left(p-2\right)}{r^2}.
 \]
 \end{itemize}
 Then every $p$-subharmonic $u\in C^1\left(\mathbb{R}^2\right)$ satisfying
 \[
 \liminf_{r\rightarrow\infty}\frac{M\left(u,r\right)}{\log\log r}=0
 \]
 is constant.
\end{theorem}
\begin{proof}
Let $R>0$ be such that if $r\geq R$ (i) holds, and
 let $z\left(r\right)=\left(r\log\left(\frac{r}{R}\right)\right)^{p-1}$. Then
 \[
 -\frac{z''\left(r\right)}{z\left(r\right)}=
-(p-1) \left[ (p-2) \left( \frac{\log\left( \frac{r}{R} \right) + 1}{r \log\left( \frac{r}{R} \right)} \right)^2 + \frac{1}{r^2 \log\left( \frac{r}{R} \right)} \right]
\leq -\frac{p-1}{r^2\log r}.
 \]
 Also
 \[
 \frac{z'\left(r\right)}{z\left(r\right)}=
 (p - 1) \cdot \frac{\log(r/R) + 1}{r \log(r/R)},
 \]
 which goes to $\infty$ as $r\rightarrow R$. Hence,
 by Sturm's comparison theorem, for all $\theta\in \mathbb{S}^1$, 
 \[
 \frac{z'(r)}{z(r)}\geq \frac{\phi_r\left(\theta,r\right)}{\phi\left(\theta,r\right)}, 
 \]
 and hence
 \[
 h\left(r\right)=\int^r 
 \frac{1}{z\left(s\right)^{\frac{1}{p-1}}}
 \, ds \sim \log\log r
 \]
 is a $p$-subharmonic function. Thus, by Theorem
 \ref{thm:comparison_1} the result follows if (i) holds.

 \medskip
 If (ii) holds for $r\geq R$, just notice that
 \begin{eqnarray*}
 -\frac{z''(r)}{z(r)}&=&
-(p-1) \left[ (p-2) \left( \frac{\log\left( \frac{r}{R} \right) + 1}{r \log\left( \frac{r}{R} \right)} \right)^2 + \frac{1}{r^2 \log\left( \frac{r}{R} \right)} \right]\\
&\leq&
-(p-1)(p-2) \left( \frac{\log\left( \frac{r}{R} \right) + 1}{r \log\left( \frac{r}{R} \right)} \right)^2\\
&\leq&
-(p-1)(p-2) \left( \frac{\log\left( \frac{r}{R} \right)}{r \log\left( \frac{r}{R} \right)} \right)^2\\
&=& -\frac{(p-1)(p-2)}{r^2},
\end{eqnarray*}
and the proof is the same as in the case presented
above.
\end{proof}

\medskip
\noindent
{\bf Remark.} 
When \(p > 2\), condition \emph{(ii)} implies condition \emph{(i)}, since
\[
\frac{(p - 1)(p - 2)}{r^2} > \frac{p - 1}{r^2 \log r}
\quad \text{for } r \gg 1.
\]
Nevertheless, we chose to state both conditions explicitly for conceptual clarity. When \(p = 2\), condition \emph{(i)} becomes exactly Milnor’s curvature criterion for harmonic functions, while condition \emph{(ii)} corresponds to the classical Cheng--Yau condition. Thus, the theorem simultaneously recovers these two foundational results as special cases of a unified nonlinear framework.

\medskip
Related to our results is the beautiful work of Holopainen in 
\cite{Holopainen1999} who provided a volumetric condition
for a complete manifold to be $p$-parabolic (that is, every bounded $p$-harmonic function is constant). Holopainen's criterion, via volume comparison theorems, can be translated into
global curvature criteria for $p$-parabolicity extending the celebrated theorem of Cheng and Yau.

\end{document}